\documentclass[12pt, english, a4paper]{amsart}

\usepackage{graphicx}
\usepackage{babel}
\usepackage{amsmath}
\usepackage{amsthm}
\usepackage{amssymb}
\usepackage{amsfonts}
\usepackage{latexsym}
\usepackage[left=2.5cm, top=2.5cm, bottom=2.5cm, right=2.5cm]{geometry}

\newtheoremstyle{thm}{}{}{\slshape}{}{\scshape}{.}{0.5em}{}
\newtheoremstyle{def}{}{}{}{}{\scshape}{.}{0.5em}{}
\newtheoremstyle{rmk}{}{}{}{}{\scshape}{.}{0.5em}{}
\newtheoremstyle{claim}{}{}{}{}{\slshape}{.}{0.5em}{}

\theoremstyle{thm}
\newtheorem{newstatement}{newstatement}
\newtheorem{lemma}[newstatement]{Lemma}
\newtheorem{theorem}[newstatement]{Theorem}
\newtheorem{corollary}[newstatement]{Corollary}
\newtheorem{proposition}[newstatement]{Proposition}

\theoremstyle{def}
\newtheorem{definition}[newstatement]{Definition}

\theoremstyle{rmk}
\newtheorem{remark}[newstatement]{Remark}

\theoremstyle{claim}

\makeatletter
\renewcommand{\section}{\@startsection%
{section}
{1}
{0mm}
{1.5\bigskipamount}
{0.5\bigskipamount}
{\centering\normalsize\sc}}

\renewcommand{\paragraph}{\@startsection%
{paragraph}
{4}
{0mm}
{\bigskipamount}
{-1.25ex}
{\normalsize\sl}}

\expandafter\let\expandafter\oldproof\csname\string\proof\endcsname
\let\oldendproof\endproof
\renewenvironment{proof}[1][\proofname]{%
  \oldproof[\slshape #1]%
}{\oldendproof}
\def\provedboxcontents#1{$\square$}

\def\fieldstyle{\Bbb} 

\def\p_#1{p_{\kern-1.5pt_#1}}
\def\emph#1{{\em #1}}
\def\textbf#1{{\bf #1}}
\def\MR#1{\relax}

\newcommand{\Cl}{\mathop{\mathrm{Cl}}\nolimits} 
\newcommand{\Int}{\mathop{\mathrm{Int}}\nolimits} 

\newcommand{\Spin}{\mathop{\mathrm{Spin}}\nolimits}

\newcommand{\crit}{\mathop{\mathrm{Crit}}\nolimits}

\newcommand{\GL}{\mathop{\mathrm{GL}}\nolimits}
\newcommand{\SO}{\mathop{\mathrm{SO}}\nolimits}
\newcommand{\U}{\mathop{\mathrm{U}}\nolimits}

\newcommand{\BSO}{\mathop{\mathrm{BSO}}\nolimits}
\newcommand{\BU}{\mathop{\mathrm{BU}}\nolimits}

\newcommand{\R}{\fieldstyle R}
\newcommand{\C}{\fieldstyle C}

\newcommand{\Z}{\Bbb Z}

\newcommand{\CP}{\C\mathrm P}

\let\hash\#
 
\renewcommand{\#}{\mathbin{\hash}}

\def\varemptyset{{\text{\raise.21ex\hbox{$\not$}}\mkern.15mu\mathrm{O}\mkern.15mu}}

\let\emptyset\varemptyset
\let\geq\geqslant
\let\leq\leqslant

\let\phi\varphi
\let\epsilon\varepsilon

\title
[Embedding almost complex manifolds]
{On embeddings of almost complex manifolds in almost complex Euclidean spaces}

\author{Antonio J. Di Scala, Naohiko Kasuya and Daniele Zuddas}
\address{\normalfont Antonio J. Di Scala, Dipartimento di Scienze Matematiche `G.L. Lagrange', Politecnico di Torino, Corso Duca degli Abruzzi 24, 10129, Torino, Italy.}
\email{antonio.discala@polito.it}
\address{\normalfont Naohiko Kasuya, Department of Social Informatics, Aoyama Gakuin University, 5-10-1 Fuchinobe, Chuo-ku, Sagamihara, Kanagawa 252-5258, Japan.}
\email{nkasuya@si.aoyama.ac.jp}
\address{\normalfont Daniele Zuddas, Korea Institute for Advanced Study, 85 Hoegiro, Dongdaemun-gu, Seoul 02455, Republic of Korea.}
\email{zuddas@kias.re.kr}

\date{}


\begin{document}

\begin{abstract}
We prove that any compact almost complex manifold $(M^{2m}, J)$ of real dimension $2m$ admits a pseudo-holomorphic embedding in $(\R^{4m+2}, \tilde{J})$ for a suitable positive almost complex structure $\tilde J$. Moreover, we give a necessary and sufficient condition, expressed in terms of the Segre class $s_m(M, J)$, for the existence of an embedding or an immersion in $(\R^{4m}, \tilde{J})$. We also discuss the pseudo-holomorphic embeddings of an almost complex 4-manifold in $(\R^6, \tilde J)$.
\end{abstract}

\maketitle

\section{Introduction}

In this article we give some existence results of pseudo-holomorphic embeddings of almost complex manifolds into almost complex Euclidean spaces. More precisely, we prove that such an embedding exists if the dimension of the ambient Euclidean space is at least $4m+2$, where $2m$ is the real dimension of the source manifold. We also give results about pseudo-holomorphic immersions and embeddings into $\R^{4m}$ under certain assumptions on the Chern class. The Euclidean space is endowed with a suitable non-standard almost complex structure, which is not integrable in general. 

As a further result, we provide a condition for the existence of codimension-two pseudo-holomorphic embeddings of almost-complex 4-manifolds.

We notice that Theorem~\ref{PHE} below represents a major improvement of the main result of \cite{DZ11}, where the pseudo-holomorphic embedding was in $\R^{6m}$. We reduce the dimension of the ambient Euclidean space to $4m+2$. It is the best possible result, since there is an obstruction for pseudo-holomorphic embeddings in $\R^{4m}$, as stated in Theorem~\ref{PHI}. 
We also fix a mistake in the proof of the main result of \cite{DZ11} that has to do with the homotopy type of the space of linear complex structures on $\R^{2n}$. These spaces have been considered to be $(n-1)$-connected in \cite{DZ11}, but this is false.
In the present paper we prove a stronger result, which is  represented by Theorem~\ref{PHE} below.

The space of positive linear complex structures on $\R^{2n}$ is homotopy equivalent to the symmetric space $\Gamma(n) = \SO(2n)/\U(n)$.
If $k\leq 2n-2$, the homotopy groups $\pi_k(\Gamma(n))$ are said to be stable.
The stable homotopy groups of $\Gamma(n)$ have been computed by Bott \cite{Bo59}, who showed that for $k\leq 2n-2$,
\begin{align*}
\pi _k (\Gamma(n)) \cong \pi_{k+1} ({\SO(2n)}) \cong
\begin{cases}
 0 & \mbox{for } k\equiv 1,3,4,5\\
\Z & \mbox{for } k\equiv 2,6\\
\Z_2 & \mbox{for } k\equiv 0,7
\end{cases}
\pmod{8}.
\end{align*}

For the unstable homotopy groups of $\Gamma(n)$, see \cite{Ma61}, \cite{Ha63}, \cite{Ka78} and \cite{Os84}.

Now we state our main results, which are the following.

\begin{theorem}~\label{PHE}
Any almost complex manifold $(M^{2m}, J)$ of real dimension $2m$ can be pseudo-holomorphically embedded
in $(\R^{4m+2}, \tilde{J})$ for a suitable positive almost complex structure $\tilde{J}$.
\end{theorem}

Notice that the codimension $2m + 2$ improves the $4m$ of \cite[Theorem 1]{DZ11} for $m > 1$.

We denote by $c(M, J)$ the total Chern class of $(M, J)$, and by $s(M, J) = c(M, J)^{-1}$ the total Segre class of $(M, J)$. Let $s_k(M, J) \in H^{2k}(M)$ be the $2k$-dimensional term of $s(M, J)$. We also define $$I(M, J) = - \frac{1}{2} \langle s_m(M, J), [M] \rangle \in \Z.$$

\begin{remark}
$I(M, J)$ is an integer because the normal Euler number of an immersion of $M^{2m}$ in $\R^{4m}$ is even by a theorem of Whitney \cite{Wh44}. Indeed, $-2I(M, J)$ is going to be the normal Euler number, see the proof of Theorem~\ref{PHI}.
\end{remark}

\begin{theorem}~\label{PHI}
An almost complex manifold $(M^{2m}, J)$ of real dimension $2m$ can be pseudo-holomorphically immersed
in $(\R^{4m}, \tilde{J})$ for a suitable positive almost complex structure $\tilde{J}$ if and only if $I(M, J) \geq 0$. In this case, there is a self-transverse pseudo-holomorphic immersion $(M, J) \looparrowright (\R^{4m}, \tilde J)$ with exactly $I(M, J)$ double points. Thus, $(M, J)$ can be pseudo-holomorphically embedded in $(\R^{4m}, \tilde{J})$, for some $\tilde J$, if and only if $I(M, J) = 0$.
\end{theorem}

The first part of the following corollary is immediate because $s_1(S, J) = -c_1(S, J)$, hence $I(S, J) = 1 - g(S)$ for a closed Riemann surface $(S, J)$, see also \cite{DSV2010}. The second part is a consequence of known facts, that is the existence of a torus fibration on $S^4$, which is due to Matsumoto \cite{Mat82}. However, we incorporate this result in the corollary because this construction enlightens nicely and more concretely our results in the case of elliptic curves as a family of pseudo-holomorphic curves in $\R^4$.

\begin{corollary}\label{pseudo/thm}
$(S, J)$ can be pseudo-holomorphically immersed in $(\R^4, \tilde J)$, for some $\tilde J$, if and only if $S$ is either a torus or a sphere. Moreover, $\tilde J$ can be suitably chosen so that $(\R^4, \tilde J)$ admits a two-dimensional holomorphic singular foliation with the property that the regular leaves but one are pseudo-holomorphically embedded tori, the other regular leaf is a pseudo-holomorphically embedded cylinder $S^1 \times \R$, and the singular leaf is a pseudo-holomorphically immersed sphere with one node.
\end{corollary}

For a closed, oriented 4-manifold $M$, we denote the signature of $M$ by $\sigma(M)$.

\begin{theorem}~\label{4-manifolds}
Suppose that $(M, J)$ is a closed almost complex $4$-manifold such that $H^2(M)$ has no $2$-torsion. 
Then, $(M, J)$ can be pseudo-holomorphically embedded in $(\R^6, \tilde{J})$, for some $\tilde{J}$, if and only if $\sigma(M) = \chi(M) = 0$ and $c_1(M, J) = 0$. 
\end{theorem}

\begin{remark} Our proof of Theorem \ref{PHE} also shows the existence of an isometric and pseudo-holomorphic embedding of 
an almost Hermitian manifold $(M^{2m},J,g)$ into $(\R^{2q_m},\tilde{J},g_0)$, where $g_0$ is the flat standard metric, 
$\tilde{J}$ is a suitable almost complex structure compatible with $g_0$ and $q_m$ is a sufficiently large integer. 
Indeed, the proof of Theorem \ref{PHE} starts with the use of the well-known theorem of Whitney to embed $M$ into $\R^{4m + 2}$.
If we start our proof of Theorem \ref{PHE} with an isometric embedding $f$ of $(M,g)$ into $(\R^{2q_m},g_0)$ where $q_m$ is determined as in Nash's theorem \cite{Na56}, then our proof shows the existence of the suitable $\tilde{J}$ compatible with $g_0$ such that $f$ is also pseudo-holomorphic (for recent improvements of $q_m$ see \cite{HH06} and the references therein).
\end{remark}

The paper is organized as follows. In Section~\ref{prelim/sec} we address some preliminaries and fix notations. In Section~\ref{proofs/sec} we prove Theorems \ref{PHE} and \ref{PHI}. The proofs make use of Proposition~\ref{key}, which is proved therein. In Section~\ref{foliat/sec} we quickly recall basic facts about Lefschetz fibrations, and prove the second part of Corollary~\ref{pseudo/thm}. 
Section~\ref{dim4/sec} addresses the four-dimensional case, with the proof of Theorem~\ref{4-manifolds}. We conclude with some remarks in Section~\ref{fin-rmk/sec}, providing also the sketch of an alternative proof of Theorem~\ref{4-manifolds}.

\section{Preliminaries}\label{prelim/sec}
Throughout this paper, manifolds are assumed to be connected, oriented, and smooth, that is of class $C^\infty$. Maps between manifolds are also assumed to be smooth, if not differently stated.

Let $M$ be an even dimensional manifold.
An automorphism $J \colon TM \to TM$ satisfying $J^2=-\mathrm{id}$ is called an almost complex structure on $M$.
The pair $(M,J)$ is called an almost complex manifold.
Equivalently, an almost complex structure is a complex structure on the tangent bundle, where the multiplication by $i$ on $T_p M$ corresponds to the action of $J_p$. In particular, an almost complex structure determines a preferred orientation on $M$.
If $M$ is already oriented, $J$ is said to be positive if the given orientation coincides with the one determined by $J$. Otherwise, $J$ is said to be a negative almost complex structure.

The space $\widetilde\Gamma(n)$ of positive linear complex structure on $\R^{2n}$, that is the set of matrices which are conjugate to
$$J_n = \bigoplus_n \begin{bmatrix}
0 & -1\\
1 & 0
\end{bmatrix}$$
by an element of $\GL_+(\R^{2n})$, is diffeomorphic with the symmetric space $\GL_+(\R^{2n})/\GL(\C^n)$, which in turn is homotopy equivalent to $\Gamma(n)$. An almost complex structure on $\R^{2n}$ can be considered as a map $J \colon \R^{2n} \to \widetilde\Gamma(n)$.

\begin{definition}
Let $(M,J)$ and $(N,J')$ be almost complex manifolds.
A map $f \colon M\to N$ is said to be pseudo-holomorphic if $Tf\circ J=J'\circ Tf$. Equivalently, $f$ is pseudo-holomorphic if and only if the tangent map is complex linear at each point.
\end{definition}

Our results depend on well-known theorems of Whitney on the existence of immersions and embeddings of smooth manifolds into Euclidean spaces \cite{Wh44, Wh44b}. We also make use of
the Hirsch-Smale theory on the classification of immersions \cite{Hi59, Sm59}. We first quickly recall Whitney's theorems.

\begin{theorem}[Whitney \cite{Wh44}]~\label{embedding}
Any smooth $m$-manifold can be embedded in $\R^{2m}$. 
\end{theorem}

\begin{theorem}[Whitney \cite{Wh44b}]
For $m>1$, any smooth $m$-manifold can be immersed in $\R^{2m-1}$. 
\end{theorem}

A key ingredient in the proof of Theorem~\ref{embedding} is the self-intersection number $I(f)$ of an immersion $f \colon M^{m}\to \R^{2m}$ with normal crossings, 
that is an immersion $f$ whose self-intersections are all transverse double points.

When $m$ is even and $M$ is orientable, $I(f)$ is the algebraic self-intersection of $f$ (whereas 
if $m$ is odd or $M$ is non-orientable, $I(f)$ is defined only mod 2). 
By results of Whitney for $m \geq 3$ and Hirsch (based on work of Smale) for $m = 2$, two immersions $f$ and $g$ of $M^{m}$ into $\R^{2m}$ are regularly homotopic if and only if $I(f) = I(g)$ and moreover any $f$ is regularly homotopic to an immersion with exactly $|I(f)|$ singular points. In particular, $f$ is regularly homotopic to an embedding if and only if $I(f) = 0$.

Our proof of Theorem~\ref{PHI} is based on the following theorem of Hirsch \cite{Hi59} and Smale \cite{Sm59}. We denote by $\mathrm{Imm}(M,\allowbreak N)$ the space of immersions of $M$ into $N$, and by $\mathrm{Mon}(TM,\allowbreak TN)$ the space of bundle monomorphisms from $TM$ to $TN$, endowed with the $C^{\infty}$-topology, with $M^m$ and $N^n$ manifolds of dimensions $m < n$.

\begin{theorem}\label{SH}
The tangent map $T \colon \mathrm{Imm}(M, N)\to \mathrm{Mon}(TM, TN)$ is a weak homotopy equivalence. 
In particular, the tangent map induces a bijection $\pi_0(\mathrm{Imm}(M, N)) \cong \pi_0(\mathrm{Mon}(TM, TN))$. 
\end{theorem}

In other words, the classification of immersions up to regular homotopy is reduced to that of monomorphisms up to homotopy. 

For $N = \R^n$, homotopy classes of bundle monomorphisms $TM \to T\R^n$ correspond to homotopy classes of maps $TM \to \R^n$ which are linear and injective on each fiber, and in turn these are sections of a bundle over $M$ with fiber diffeomorphic to the Stiefel manifold $V_m(\R^n)$ of linear injective maps $\R^m \to \R^n$, that is the space of $n \times m$ matrices of rank $m$.
Notice that $V_m(\R^n)$ is homotopy equivalent to the homogeneous space $\SO(n)/\SO(n-m)$.
   
Now we focus on the immersions of a closed, oriented, connected manifold $M^{2m}$ into $\R^{4m}$. 
By obstruction theory, the obstructions to homotopy between two sections of a $V_{2m}(\R^{4m})$-bundle lie in the cohomology groups $H^i(M, \pi_{i}(V_{2m}(\R^{4m})))$. 
Since $V_{2m}(\R^{4m})$ is $(2m-1)$-connected, the only obstruction lies in 
$$H^{2m}(M; \pi_{2m}(V_{2m}(\R^{4m})))\cong \pi_{2m}(V_{2m}(\R^{4m})) \cong \Z .$$
Hence, there are identifications $$\pi _0(\mathrm{Imm}(M, \R^{4m})) \cong \pi_0(\mathrm{Mon}(TM, T\R^{4m}))\cong \pi_{2m}(V_{2m}(\R^{4m})) \cong \Z.$$ This is given exactly by the self-intersection $I(f)$ for immersions with only normal crossings.

Let $\nu _f$ be the normal bundle of the immersion $f$. 
It is well known that the normal Euler class $e(\nu _f)$ is given by $-2I(f) [M]$ (see \cite{Wh44}, \cite{LS59}). 
Hence, the normal Euler class also classifies regular homotopy classes of immersions of $M^{2m}$ into $\R^{4m}$. 
We use this fact in the proof of Theorem~\ref{PHI}. 

\section{The proofs of theorems \ref{PHE} and \ref{PHI}}\label{proofs/sec}

First, we prove the following proposition. Let $J_1$ be the standard complex structure on $\C^n$. 

\begin{proposition}~\label{key}
$(M^{2m},J)$ can be pseudo-holomorphically embedded in $(\R^{2n}, \tilde{J})$ for a suitable positive almost complex structure $\tilde{J}$,
if and only if there is an embedding $f \colon M \to \R^{2n}$ such that the tangent map $Tf \colon TM \to T\R^{2n}$ is homotopic to a complex linear monomorphism (with respect to $J$ and $J_1$) through a homotopy of monomorphisms covering $f$.
\end{proposition}

\begin{proof} We begin with the `only if' part. Take a pseudo-holomorphic embedding $f \colon (M,\allowbreak J) \to (\R^{2n}, \tilde J)$. 
We can assume that $\tilde {J}$ is equal to $J_1$ outside of a sufficiently large ball $B^{2n}(R)$ of radius $R > 0$ and centered at the origin. We take $R$ such that $f(M) \subset B^{2n}(R/2)$.

Since $\tilde {J}$ and $J_1$ are both positive almost complex structures on the contractible space $\R^{2n}$, there exists a homotopy $(J_t)_{t \in [0,1]}$ of almost complex structures between $J_0=\tilde{J}$ and $J_1$ whose support is in $B^{2n}(R)$. 

In other words, we have a smooth map $g:B^{2n}(R)\times [0,1]\to \widetilde{\Gamma} (n)$, such that $g_t=g(\cdot, t) = J_t$ for all $t \in [0,1]$. Since $g_1$ is a constant map, we have the trivial lift 
$\bar{g_1}: B^{2n}(R)\to \GL_+(\R^{2n})$, which is the constant map to the identity element. 
Hence, by the homotopy lifting property, there exists a lift $$\bar{g}:B^{2n}(R)\times [0,1]\to \GL_+(\R^{2n})$$ of $g$ with respect to the projection $\GL_+(\R^{2n})\to \widetilde{\Gamma} (n)$. 
By taking the columns of the corresponding matrix, we obtain vector fields $\{e_{1}(t), e_2(t), \dots, e_{2n-1}(t), e_{2n}(t)\}$ on $\R^{2n}$, which depend smoothly on $t \in [0, 1]$, span pointwise the tangent bundle $T\R^{2n}$, and satisfy 
$$J_t(e_{2k-1}(t))=e_{2k}(t), \; e_{2k-1}(1)=\frac{\partial}{\partial x_k},\; e_{2k}(1)=\frac{\partial}{\partial y_k}\; (k=1,\dots, n),$$
where we denote by $(x_1, y_1,\dots, x_n, y_n)$ the cartesian coordinates of $\R^{2n}$.

Now, we have a homotopy $\psi_t \colon (T\R^{2n}, J_0) \to (T\R^{2n}, J_t)$ of $\C$-linear bundle isomorphisms over the identity map $\R^{2n}\to \R^{2n}$ such that $\psi_t(e_{i}(0)) = e_{i}(t)$ for all $i = 1, \dots, 2n$. Then, $F_t = \psi_t \circ Tf \colon TM \to T\R^{2n}$ is a homotopy of monomorphisms covering $f$ such that $F_0 = Tf$ and $F_1$ is complex linear with respect to the standard structure $J_1$. \\

Next, we prove the `if' part.
Let $(F_t \colon TM \to T\R^{2n})_{t\in [0,1]}$ be a homotopy of monomorphisms covering $f$ between $F_0 = Tf$ and a $\C$-linear monomorphism $F_1 \colon (TM, J) \to (T\R^{2n}, J_1)$. Let $\nu_t$ be the normal bundle of $F_t$.

The fiber of $\nu_1$ over $p \in M$ is the orthogonal complement of $F_1(T_p M)$ in $T_{f(p)}\C^n$, which is a complex linear subspace. Therefore, $\nu_1$ admits a complex structure $J'$. Since we have a homotopy $(\nu_t)_{t\in[0,1]}$ of normal bundles, we can pull back $J'$ to $\nu _0$, the normal bundle of $f$.  
Thus, we get a complex structure $J'' = J \oplus J'$ on the trivial bundle $(T\R^{2n})_{|M}$, which is a map $J'' \colon f(M)\to \widetilde\Gamma(n)$.
By construction, $J''$ is null-homotopic.
Hence, there is an extension $\tilde J \colon \R^{2n}\to \widetilde\Gamma(n)$ of $J''$, and this concludes the proof.
\end{proof}

Now, we are ready to prove Theorem~\ref{PHE}.

\begin{proof}[Proof of Theorem~\ref{PHE}]
We take an embedding $f \colon M^{2m}\to \R^{4m+2}$.
By Proposition~\ref{key}, it is enough to show
the existence of a complex linear monomorphism which is homotopic to the tangent map $Tf$ via a homotopy of monomorphisms covering $f$.
A complex linear monomorphism covering $f$ corresponds to a section of a $V_m(\C^{2m+1})$-bundle over $M$, where we denote by $V_m(\C^n) \simeq \mathrm{U}(n)/\mathrm{U}(n-m)$ the complex Stiefel manifold of complex linear injective maps $\C^m \to \C^n$. Notice that the identification $\C^n \cong \R^{2n}$ induces a canonical inclusion $V_m(\C^n) \subset V_{2m}(\R^{2n})$.

Since $V_m(\C^{2m+1})$ is $(2m+2)$-connected, a section uniquely exists up to homotopy.
Since $V_{2m}(\R^{4m+2})$ is $(2m+1)$-connected, the homotopy type of monomorphisms covering $f$ is also unique.
Hence, a pseudo-holomorphic monomorphism exists and it is homotopic to the tangent map $Tf$ via a homotopy of monomorphisms covering $f$.
\end{proof}

\begin{proof}[Proof of Theorem~\ref{PHI}]
We begin with the proof of the `only if' part. Let $f \colon (M^{2m}, J) \to (\R^{4m}, \tilde J)$ be a self-transverse pseudo-holomorphic immersion. Then the complex bundle $TM \oplus \nu_f$ is trivial. This implies that the total Chern class satisfies $c(M, J)\, c(\nu_f) = 1$. The Euler class is given by $e(\nu_f) = c_m(\nu_f)$ and coincides with the $2m$-dimensional term of $c(M, J)^{-1} = s(M, J)$. By Whitney's formula stated at the end of previous Section, $\langle -e(\nu_{f}), [M] \rangle = 2 I(f)$, which in our complex case must be non-negative. Hence, $I(M, J) \geq 0$. Moreover, if $f$ is an embedding, we get $I(M, J) = 0$.

Next prove the `if' part.
Since $V_m(\C^{2m})$ is $2m$-connected, there is a complex linear monomorphism $F \colon T M \to T\C^{2m}$.
By Theorem~\ref{SH} there is an immersion $f' \colon M \to \R^{4m}$, with normal crossings, such that $Tf'$ is homotopic to $F$ as a real monomorphism.

In particular, the normal bundle $\nu_{f'}$ admits a complex structure $J'$.
Notice that the Whitney sum $TM \oplus \nu _{f'}$ is a trivial complex vector bundle because, by construction, it is equivalent to the pullback ${(f')^*}(T\C^{2m})$. Therefore, we have $c(\nu_{f'}) = c(M, J)^{-1}$.

It follows that $-2 I(M, J) = \langle e(\nu_{f'}), [M] \rangle$. Hence, $I(M, J)$ is the algebraic self-in\-ter\-sec\-tion of $f'$ \cite{Wh44}. This implies that $f'$ is regularly homotopic to an immersion $f \colon M \to \R^{4m}$ with exactly $I(M, J)$ double points, which are necessarily all positive.
Notice that the normal bundle $\nu_f$, being isomorphic to $\nu_{f'}$,  inherits a complex structure which we still denote by $J'$.

\medskip

If $I(M, J) = 0$, $f$ is an embedding and the complex bundle $TM \oplus \nu_f$ over $M$ gives a map $g \colon f(M) \to \widetilde\Gamma(2m)$ which is homotopic to a constant, since $Tf \colon TM \to T\R^{4m}$ is homotopic to a complex linear monomorphism by a homotopy of linear monomorphisms. Therefore, $g$ can be extended to an almost complex structure $\tilde J$ on $\R^{4m}$ such that $f \colon (M, J) \to (\R^{4m}, \tilde J)$ is a pseudo-holomorphic embedding.

\medskip

Next consider the case $I(M, J) > 0$.
First, isotope $f$ so that the two tangent spaces at the self-intersection points of $f(M)$ are orthogonal.
We regard $(T_x M)^\perp$ as a linear subspace of $T_{f(x)} \R^{4m}$, endowed with the complex structure $J'_x$. By means of the linear monomorphism $T_x f \colon T_x M \to T_{f(x)} \R^{4m}$, we can also identify $T_x M$ with a linear subspace of $T_{f(x)} \R^{4m}$, for all $x \in M$.

Let $p, q \in M$ be two distinct points such that $f(p) = f(q)=r$. 
By the connectedness of $\widetilde \Gamma(m)$, we can homotope $J'$ in a neighborhood of $p$ and $q$ so that $J'_p = J_q$ and $J'_q=J_p$ (up to the above identification). Indeed, this can be achieved by changing $J'$ inside disjoint balls around such points $p$ and $q$.
 
After performing this homotopy on all such pairs of points, we get an almost complex structure $J \oplus J'$ on $\R^{4m}$ which is well-defined along the immersed manifold $f(M)$, that is a map $g \colon f(M) \to \widetilde \Gamma(2m)$. 

The composition $g\circ f: M\to \widetilde \Gamma(2m)$ is null-homotopic because, by construction, $(\nu_f, J')$ is homotopic to $\nu_{F}$, and so $g \circ f$ is homotopic to the pullback of $T\C^{2m}$ by $F$, which is a constant map.

The space $f(M)$ is obtained from $M$ by identifying finitely many pairs of points. So, $f(M)$ is homotopy equivalent to the wedge sum of $M$ with $n$ copies of $S^1$, that is $f(M) \simeq M \vee (\vee_n S^1)$, where $n$ is the number of the double points of $f$. In order to show that $g$ is homotopic to a constant, we can assume that $g$ is defined on this wedge sum (up to composing $g$ with a suitable homotopy equivalence).

Since $\widetilde\Gamma(2m)$ is simply-connected, the map $g$ can be homotoped to a constant on the circles $S^1$ in the wedge sum. So, $g$ factorizes by a map $g' \colon M \to \widetilde \Gamma(2m)$ which is homotopic to $g \circ f$, hence to a constant. Therefore, $g$ is homotopic to a constant. This means that $g$ can be extended to a map $\tilde J \colon \R^{4m} \to \widetilde \Gamma(2m)$, which is the desired almost complex structure.
\end{proof}

\begin{remark}
In the case where $M^{2m}$ is open,
any almost complex manifold $(M^{2m},J)$ can be pseudo-holomorphically embedded
in $(\R^{4m}, \tilde{J})$ for a suitable positive almost complex structure $\tilde{J}$.
Since the open manifold $M^{2m}$ is isotopic to a neighborhood of the $(2m-1)$-skeleton and $V_{2m}(\R^{4m})$ is $(2m-1)$-connected,
the space $\mathrm{Mon}(TM^{2m},T\R^{4m})$ is path-connected.
\end{remark}

\section{A pseudo-holomorphic foliation of $(\R^4, \tilde J)$ and the proof of Corollary~\ref{pseudo/thm}}\label{foliat/sec}

We first recall the notion of Lefschetz fibration. For further details and general facts about Lefschetz fibrations, see for example \cite{GS99} or \cite[Section 6]{APZ2013}.

Let $M$ be a closed 4-manifold, and let $S$ be a closed surface.
A Lefschetz fibration on $M$ is a map $f \colon M \to S$ such that at the critical points, $f$ is locally equivalent to the complex non-degenerate quadratic form $(z_1, z_2) \mapsto z_1 z_2$ (positive critical point), or to $(z_1, z_2) \mapsto z_1 \bar z_2$ (negative critical point), with respect to suitably chosen local complex coordinates that are compatible with the given orientations. It follows that a Lefschetz fibration is an open map. We assume that $f$ is injective on the critical set, which is a finite set.

Away from the critical image $\crit(f) \subset S$, $f$ is a surface bundle over $S - \crit(f)$, with fiber a closed, oriented surface $F$ (the regular fiber of $f$). If $g$ is the genus of $F$, we say that $f$ is a genus-$g$ Lefschetz fibration.

A singular fiber $F_a = f^{-1}(a)$, $a \in \crit(f)$, is an immersed surface with one node singularity. Notice that this node can be positive or negative, accordingly with its sign as a critical point of $f$. We denote by $\crit_+(f) \subset S$ the set of positive critical values of $f$, and by $\crit_-(f) \subset S$ the set of negative critical values. Hence, we have $\crit(f) = \crit_+(f) \sqcup \crit_-(f)$.
By looking at the local model, one can show that there is a simple curve $c_a \subset F$, which is called a vanishing cycle, such that $F_a \cong F / c_a$. The monodromy of a small loop around $a$ in $S$ is a Dehn twist about the curve $c_a$. The vanishing cycles are not uniquely determined, depending on the choice of generators for $\pi_1(S - \crit(f))$.

Given a Lefschetz fibration $f \colon M \to S$, it is a known fact that there are almost complex structures $\tilde J$ on $M_+ = M - \crit_-(f)$ and $J$ on $S$ such that $f_| \colon (M_+, \tilde J) \to (S, J)$ is pseudo-holomorphic. In particular the fibers of $f_{|M_+}$ are pseudo-holomorphic (possibly immersed) curves, which define a pseudo-holomorphic singular foliation.
For the sake of completeness, we give the construction of such almost complex structures.

For any positive critical point $a$ of $f$, take a local complex chart $(U_a, \phi_a)$ around $a$ and a local complex chart $(V_a, \psi_a)$ around $f(a)$ such that $(\psi_a \circ f\circ \phi^{-1}_a)(z_1, z_2) = z_1 z_2$. Moreover, we assume that $V_a = f(U_a)$ and $\Cl(V_a) \cap \Cl(V_{a'}) = \emptyset$ for all $a \neq a'$. Take also a smaller compact neighborhood $U'_a \subset U_a$ of $a$.

Next, we endow $M$ with a Riemannian metric $g$ such that $\phi_a$ is an isometry in a neighborhood of $U'_a$, for all $a \in \crit_+(f)$, where $\C^2$ is endowed with the Euclidean metric.

Also, endow $S$ with a complex structure $J$.
Since the space $\widetilde \Gamma(1)$ is contractible, up to deforming $J$ in a neighborhood of $f(U'_a)$, we can assume that $\psi_a \colon (\Int(f(U'_a)), J) \to \C$ is complex analytic for all $a$.

Next, define an almost complex structure $\tilde J$ on $M_+ - \crit_+(f)$ such that $\tilde J$ is a $90^\circ$-rotation in the counterclockwise direction (with respect the the given orientation) on any tangent plane $T_p F_p$ to a fiber $F_p = f^{-1}(f(p))$ for $p \in M_+ - \crit_+(f)$, while on its orthogonal complement $(T_p F_p)^\perp$, $\tilde J$ is the pullback of $J$ through the isomorphism $(T_p f)_| \colon (T_p F_p)^\perp \to T_{f(p)} S$.

It follows that  $\tilde J$ coincides in $U'_a - \{a\}$ with the integrable complex structure determined by the local chart $\phi_a$, hence $\tilde J$ extends over $\crit_+(f)$.

By construction, $f \colon (M_+, \tilde J) \to (S, J)$ is pseudo-holomorphic, hence the fibers are pseudo-holomorphic curves.

\paragraph{Matsumoto's fibration on $S^4$}
In \cite{Mat82} Matsumoto constructed a genus-1 Lefschetz fibration $f \colon S^4 \to S^2$ with two critical points $\{a_+, a_-\}$ of opposite signs (see also \cite[Example 8.4.7]{GS99} for a description in terms of Kirby diagrams).

We sketch this construction here. We begin with the Hopf fibration $h \colon S^3 \to S^2$, which is the projectivization $\C^2 - \{0\} \to \CP^1 \cong S^2$ restricted to $S^3 \subset \C^2$, which is defined by the equation $|z_1|^2 + |z_2|^2 = 1$. Then make the suspension $\varSigma h \colon S^4 \to S^3$. The map $f$ can be described topologically as a perturbation of $f_1 = h \circ \varSigma h$. It is worth noting that both $f_1$ and $f$ are nice maps that represent the generator of $\pi_4(S^2) \cong \Z_2$.

The construction of $f$ goes as follows.
Recall that the Hopf fibration is given by $$h(z_1, z_2) = (2 z_1 \bar z_2, |z_1|^2 - |z_2|^2),$$ where we consider $S^2 \subset \C \times \R$ to be defined by the equation $|z|^2 + x^2 = 1$.

The suspension can be represented by $$\varSigma h(z_1, z_2, x) = (2 z_1 \bar z_2, |z_1|^2 - |z_2|^2 + i x \sqrt{2 - x^2}),$$ where $S^4 \subset \C^2 \times \R$ has equation $|z_1|^2 + |z_2|^2 + x^2 = 1$.

After some straightforward simplifications, we get that $f_1 = h \circ \varSigma h$ can be expressed by the formula 
$$f_1(z_1, z_2, x) = (4 z_1 \bar z_2 (|z_1|^2 - |z_2|^2 - i x \sqrt{2 - x^2}), 8 |z_1|^2 |z_2|^2 - 1).$$

It can be proved that $f_1$ is a genus-1 Lefschetz fibration with two critical points $a_\pm = (0, 0, \pm 1) \in S^4$ of opposite signs. The problem is that $f_1(a_+) = f_1(a_-)$, so the singular fiber is a sphere with two opposite nodes. This is due to the fact that $\varSigma h(a_+) = (0, i)$ and $\varSigma h(a_-) = (0, -i)$ belong to the same fiber of $h$. We handle this issue by taking an orientation-preserving diffeomorphism $k \colon S^3 \to S^3$ such that $k(0, i)$ and $k(0, -i)$ belong to different fibers of $h$. Finally put $$f = h \circ k \circ \varSigma f.$$ It follows that $f$ is a genus-1 Lefschetz fibration with two critical points $a_\pm$ such that $f(a_+) \neq f(a_-)$.

The singular fibers of $f$ are two immersed spheres $\Sigma_\pm = f^{-1}(f(a_\pm))$, each one with one positive or negative node singularity. Notice that $\Sigma_\pm - \{a_\pm\} \cong S^1 \times \R$.

By the above construction, we get an almost complex structure $\tilde J$ on $S^4 - \{a_-\} \cong \R^4$, such that $f_| \colon (\R^4, \tilde J) \to S^2$ is pseudo-holomorphic, and this concludes the proof of Corollary~\ref{pseudo/thm}.

\begin{remark}
In \cite{DKZ2015} we show by different methods that $\tilde J$ can be taken integrable.
\end{remark}

\section{Almost complex 4-manifolds in $\R^6$}\label{dim4/sec}

In this section we prove Theorem~\ref{4-manifolds}.
We will make use of the following theorem.

\begin{theorem}[Cappell-Shaneson \cite{CS79, Ru82}]~\label{CS}
A closed, orientable, smooth $4$-manifold $M$ embeds in $\R^6$ if and only if $w_2(M) = 0$ and $\sigma(M) = 0$.  
\end{theorem}

\begin{proof}[Proof of Theorem~\ref{4-manifolds}]
We begin with the `only if' part. 
Take a pseudo-holomorphic embedding $f \colon (M, J)\to (\R^6, \tilde{J})$. 
Since $f$ is a codimension-2 embedding in a Euclidean space, the normal complex line bundle $\nu_f$ is trivial (see for example \cite[Sec. 8, Theorem 2]{Ki89}), so $$c(M, J) = c(M, J)\, c(\nu_f) = f^*(c(\R^6, \tilde{J})) = 1.$$ 
Therefore, $c_1(M, J)=0$ and $\chi(M)= \langle c_2(M, J), [M]\rangle = 0$. Moreover, $w_2(M) = 0$ and $\sigma (M) = 0$ by Theorem~\ref{CS}.

Next, we prove the `if' part. Since $\sigma(M) = 0$ and $w_2(M) \equiv c_1(M, J) \pmod 2$ is zero, there is an embedding $f \colon M\to \R^6$ by Theorem~\ref{CS}. 
The normal bundle $\nu_f$ is trivial \cite{Ki89}. Let $J'$ be a complex structure on $\nu_f$ such that $J'$ is compatible with the normal orientation induced by the embedding $f$.

Thus, we get an almost complex structure $J\oplus J'$ on $(T\R^6)_{|M}\cong TM \oplus \nu_f$, that is a map $g \colon M \to \widetilde\Gamma(3)$. 
Now we show that $g$ is null-homotopic. Take a cell decomposition $M^{(0)}\subset M^{(1)}\subset M^{(2)}\subset M^{(3)}\subset M^{(4)}=M$, where $M^{(i)}$ is the $i$-skeleton of $M$.
 
Recall that 
\begin{align*}
\pi_{i}(\widetilde\Gamma(3))=\pi_{i}(\Gamma (3))=
\begin{cases}
 0 & \mbox{for } i= 1,3,4 \\
\Z & \mbox{for } i= 2.
\end{cases}
\end{align*}
Therefore, the only obstruction $\Omega(g)$ for $g$ to be null-homotopic lies in $$H^2(M; \pi_2(\Gamma(3))) = H^2(M).$$
In other words, $g$ is homotopic to a constant map over the $2$-skeleton $M^{(2)}$ if and only if the element $\Omega(g) \in H^2(M)$ is zero. Moreover, once $g$ has been homotoped to a constant on $M^{(2)}$, we can extend this homotopy over higher skeleta because $\pi_3(\Gamma(3)) = \pi_4(\Gamma(3)) = 0$.

Since $c_1(M, J) = 0$ and since $H^2(M)$ has no 2-torsion, the following lemma implies that $\Omega(g) = 0$, and this concludes the proof.
\end{proof}

\begin{lemma}
$2\Omega(g) = c_1(M, J)$.
\end{lemma}

\begin{proof}
For a topological group $G$, we denote by $BG$ the classifying space of $G$, so that there is a natural bijection between the set of equivalence classes of principal $G$-bundles over any CW-complex $X$, and the set of homotopy classes of maps $[X, BG]$, see \cite{Mi56a, Mi56b}.

The inclusion $\U(3)\to\SO(6)$ induces a fibration $\BU(3)\to \BSO(6)$ with fiber $\Gamma (3)$ \cite[p. 680]{Wh78}. 
Let $\iota \colon \Gamma(3)\to \BU(3)$ be the inclusion. 
Then the composition $\iota \circ g \colon M\to \BU(3)$ is the classifying map of the complex vector bundle $TM\oplus \nu _f$.

The homotopy exact sequence for the fibration $\Gamma (3)\to \BU(3)\to \BSO(6)$, that is
$$\pi_3(\BSO(6))\to \pi_2(\Gamma(3)) \to \pi_2(\BU(3)) \to \pi_2(\BSO(6))\to \pi_1(\Gamma(3))$$ is given by $0\to \Z \to \Z \to \Z_2 \to 0$. 
Hence, the map $\iota_{\ast } \colon \pi_2(\Gamma(3)) \to \pi_2(\BU(3))$ is the double map in $\Z$. 
Thus, $\Omega $ is mapped to $$2\Omega \in H^2(M;\pi_2(\BU(3)))=H^2(M),$$ which is the obstruction for $\iota \circ g$ to be null-homotopic over $M^{(2)}$. 

On the other hand, $(\iota \circ g)_{|M^{(2)}}$ determines the class $(\iota \circ g)^{\ast }c_1$, where $c_1$ is a generator of $H^2(\BU(3))\cong \Z$.    
Hence, $2\Omega(g)$ is equal to the first Chern class of the complex vector bundle $TM\oplus \nu _{f}$. 
Therefore, we obtain $2 \Omega(g) =c_1(M, J\oplus J')=c_1(M, J)$. 
We note that a similar argument can be found in the proof of Theorem 8.18 in \cite{Ha08} and in Section 8.1 of \cite{Ge08}. 
\end{proof}

\section{Final remarks}\label{fin-rmk/sec}

\begin{enumerate}
\item In the proofs of Theorems~\ref{PHE} and \ref{4-manifolds} we actually showed that any given embedding $f$ of $M$ into the ambient Euclidean space can be made pseudo-holomorphic with respect to a suitable almost complex structure on the ambient space (that is, $\R^{4m+2}$ or $\R^6$). The same holds for a given immersion into $\R^{4m}$ in the proof of Theorem~\ref{PHI} with minimal self-intersection, provided it is self-orthogonal at the self-intersection points.

\item By a theorem of Dold and Whitney \cite{DW59} and by Hirzebruch's signature formula \cite{Hi66}, 
we have that $w_2(M)=0$ and $\sigma (M)=\chi (M)=0$ if and only if $M$ is parallelizable. 

In general, if $M^{2m}$ has trivial Euler characteristic and embeds in $\R^{2m+2}$, then $M^{2m}$ is parallelizable. 
This follows from Kervaire's theorem about the generalized curvatura integra \cite{Ke57}. We quickly sketch this proof.

Let $M^{n}\subset \R^{n+k}$, $k > 1$, be a framed submanifold, that is an embedded submanifold $M$ along with a trivialization of the normal bundle, which is essentially a map $G \colon M \to V_{k}(\R^{n+k})$ (the {\sl generalized Gauss map}). 
Since $V_{k}(\R^{n+k})$ is $(n-1)$-connected, the element
\begin{align*}
G_{\ast}([M])\in H_{n}(V_{k}(\R^{n+k}))\cong \pi_{n}(V_{k}(\R^{n+k})) \cong
  \begin{cases}
  \Z & n\ \text{even}\\
  \Z_{2} & n\ \text{odd,}
  \end{cases}
\end{align*}
which is called the {\sl generalized curvatura integra,} is the only obstruction for $G$ to be null-homotopic.

In \cite{Ke57}, Kervaire expressed $G_{\ast}([M])$ in terms of 
the Hopf invariant of the framed submanifold $M$, 
the Euler characteristic $\chi (M)$, 
and the Kervaire semi-characteristic of $M$. 
In particular, for $n = 2m$, $G_{\ast}([M]) = \frac{1}{2}(\chi (M))$. 

Therefore, for a real oriented codimension-2 submanifold $M^{2m} \subset \C^{m+1}$ with $\chi (M) = 0$, we have that the generalized curvatura integra $G_{\ast}([M])$ is zero for any trivialization of the normal bundle. 
This implies that $G \colon M \to V_{2}(\R^{2m+2})$ is homotopic to 
the constant map given by the complex vector field $\partial \slash \partial z_{m+1}$, where $(z_1,\dots ,z_{m+1})\in \C^{m+1}$ are the standard coordinates. 
Hence, $T M$ is homotopic, as a subbundle of $T \C^{m+1}$, to the trivial bundle $M \times \C^m$.
Therefore, $M^{2m}$ is parallelizable.

These considerations lead to an alternative proof of Theorem~\ref{4-manifolds}. Let $(M, J)$ be an almost complex 4-manifold which satisfies the hypotheses of Theorem~\ref{4-manifolds}. So, we can consider $M$ as a real submanifold of $\C^3$, having trivial curvatura integra.

It follows that a trivialization of the normal bundle is homotopic to the vector field $\partial \slash \partial z_{3}$, and the tangent bundle is homotopic to $M \times \C^2$ in $(T\C^3)_{|M} = M \times \C^3$. This way, $M$ inherits another almost complex structure $J_1$ induced by the identification $TM \cong M \times \C^2$.

The assumption that $H^2(M)$ has no $2$-torsion is equivalent to the condition that spin$^{c}$ structures on $M$ are classified by the first Chern class (see for example Theorem 2.4.9 in \cite{GS99}). 
A spin$^{c}$ structure on $M$ can be identified with the homotopy class of a complex structure over the $2$-skeleton $M^{(2)}$ that extends over the $3$-skeleton $M^{(3)}$ \cite{Go97}. 

It follows that the two almost complex structure $J$ and $J_1$ on $M$ are homotopic over $M^{(2)}$  
because both of them have trivial first Chern class.

Let $J'$ be a complex structure on the normal bundle of $M$.
We have that the standard complex structure on the bundle $(T\C^3)_{|M}$ is homotopic, over $M^{(2)}$, with the Whitney sum $J\oplus J'$. 
Therefore, the obstruction $\Omega(g)$ defined in the previous Section is trivial, implying that the map $g \colon M \to \widetilde\Gamma(3)$ is null-homotopic.   
\end{enumerate}

\section*{Acknowledgments}

A.J. Di Scala and D. Zuddas are members of the group GNSAGA of INdAM, Italy.

\let\emph\textsl

\end{document}